\documentclass[12pt,sn-mathphys-ay]{sn-jnl}

\usepackage{graphicx}%
\usepackage{multirow}%
\usepackage{amsmath,amssymb,amsfonts,}%
\usepackage{mathtools}
\usepackage{amsthm}%
\usepackage{mathrsfs}%
\usepackage[title]{appendix}%
\usepackage{xcolor}%
\usepackage{textcomp}%
\usepackage{manyfoot}%
\usepackage{booktabs}%
\usepackage{algorithm}%
\usepackage{algorithmicx}%
\usepackage{algpseudocode}%
\usepackage{listings}%
\usepackage{hyperref}
\usepackage{bm}
\usepackage{dsfont}
\usepackage[labelfont=bf,labelsep=space]{caption}
\usepackage{enumitem}
\theoremstyle{thmstyleone}%
\newtheorem{theorem}{Theorem}
\newtheorem*{theorem*}{Theorem}

\theoremstyle{thmstyletwo}%
\newtheorem{remark}{Remark}%

\theoremstyle{thmstylethree}%
\begin{document}
\newcommand{\R}{\mathds{R}}
\newcommand{\So}{\mathcal{S}}
\newcommand{\Ha}{\mathcal{H}}
\newcommand{\M}{M^{+}(\R^n)}
\newcommand{\wpp}{\mathbf{W}_{0}^{1,p}(\Omega)}
\newcommand{\wa}{\mathbf{W}_{\alpha,p}}
\newcommand{\wao}{\mathbf{W}_{\alpha,p}}
\newcommand{\wat}{\mathbf{W}_{\alpha,p}}
\newcommand{\wak}{\mathbf{W}_{k}}
\newcommand{\wah}{\mathbf{W}_{\frac{2k}{k+1},k+1}}
\newcommand{\ia}{\mathbf{I}_{2\alpha}}
\newcommand{\Om}{\Omega}
\newcommand{\iom}{\int_{\Omega}}
\newcommand{\e}{\epsilon}
\newcommand{\g}{\gamma}
\newcommand{\al}{\alpha}
\newcommand{\un}{u_{n}}
\newcommand{\vf}{\varphi}
\newcommand{\wc}{\rightharpoonup}
\newcommand{\phik}{\Phi^{k}(\Om)}
\newcommand{\la}{\lambda}
\newcommand{\Linf}{\Delta_{\infty}}
\newcommand{\lap}{\Delta_{p}}
\newcommand{\dx}{\mathrm{dx}}
\newcommand{\Div}{\mathrm{div}}
\def\norma#1#2{\|#1\|_{\lower 4pt \hbox{$\scriptstyle #2$}}}
\def\L#1{L^{#1}(\Omega)}

\title{An inhomogeneous p-laplacian equation with a Hardy potential}
\author{\fnm{Genival} \sur{da Silva}\footnote{email: gdasilva@tamusa.edu, website: \url{www.gdasilvajr.com}}}
\affil{\orgdiv{Department of Mathematics}, \orgname{Texas A\&M University - San Antonio}}
\abstract{In this work we study the existence and regularity of solutions to the following equation: $$\lap u + g(x) u = \frac{\la}{|x|^{p}} |u|^{p-2}u + f,$$ where $1< p < N$ and $f\in\L m$, where $m\ge 1$.}




\maketitle 
\section{Introduction}
In a recent paper \cite{boc23}, building up on the previous work \cite{boc06}, the authors study the existence and summability of the problem:
\begin{equation} 
\begin{cases}
 -\Div(M(x)u)+ g(x)u= \frac{\la}{|x|^{p}}u + f \qquad & \mbox{in } \Omega,\\
u (x) = 0 & \mbox{on }  \partial \Omega,
\end{cases}
\end{equation}
where $M(x)$ is positive matrix such that $M(x)\xi\cdot\xi>\al |\xi|^{2}$, $g(x)\ge 0$, $0<\la< \al(\frac{N-2}{2})^{2}$ and $f\in\L 1$.

In these notes we will study the equivalent problem for the p-Laplacian, namely:
\begin{equation} \label{main}
\begin{cases}
 -\lap u + g(x)u= \frac{\la}{|x|^{p}} |u|^{p-2}u + f \qquad & \mbox{in } \Omega,\\
u (x) = 0 & \mbox{on }  \partial \Omega,
\end{cases}
\end{equation}
for $1<p<N$. The crucial component of our analysis is the Hardy inequality:
\vspace{0.1in}
\begin{theorem*}(Hardy's inequality)
If $v\in\wpp$ and $1<p<N$ then
\begin{equation} \label{hardy}
\Ha^{p} \iom \frac{|v|^{p}}{|x|^{p}} \le \iom |Dv|^{p},
\end{equation}
where the constant $\Ha=\left( \frac{N-p}{p}\right)$ is optimal.
\end{theorem*}
The idea of all the proofs is essentially the same: truncate and obtain boundedness estimates. It's remarkable that we can apply the same methods of the linear case to this quasilinear case, despite the nonlinearities. 

The upshot is that Hardy's inequality provides a way to control the nonlinearity in terms of the Sobolev norm, which is essentially what we are trying to estimate.

The main challenges of problem \eqref{main} are the low summability of the source $f$ which puts the right hand side outside the dual of $\wpp$, and the nonlinearity with superlinear growth in case $p>2$ (for more on this type of problem see \cite{silva24}), which poses an obstacle to existence and high regularity.

The idea of truncating an equation in order to obtain estimates is not new, it was consolidated by Stampacchia \cite{sta65}, who proved an equivalent of Calderon-Zygmund $L^{p}$-estimates for equations with discontinuous coefficients. However, the method is not invincible, in some cases the existence of solutions is not guaranteed, even for bounded sources; nevertheless, in these notes this method will be successful due to the pivotal hole of Hardy's inequality in obtaining the estimates.

The paper is organized as follows: In section \ref{cs1}, we analyze the case with no lower order term, where we analyze two scenarios: high summability of the source, meaning $f$ is in the dual of $\wpp$; and low summability of the source, where we take the source in $\L m$ with $m>1$. In section \ref{cs3} we study the same problem with a added lower order term, we are able to obtain under certain hypothesis existence and regularity of the solutions.
\subsection*{Notation}\label{notation}
\begin{itemize}
\item[-] $\Om\subset \R^{N}$ is a bounded domain.
\item[-] The space $\wpp$ denotes the usual Sobolev space which is the closure of $\mathcal{C}^{\infty}_{0}(\Om)$, smooth functions with compact support, in the $p$-norm. 
\item[-] For $1<p<\infty$, the \textit{p-Laplacian} $\Delta_{p}$ is given by $-\Div (|Du|^{p-2}Du)$. 
\item[-] For $q>0$, $q'$ denotes the Holder conjugate, i.e. $\frac{1}{q}+\frac{1}{q'}=1$, and $q^{*}$ denotes the Sobolev conjugate, defined by $q^{*}=\frac{qN}{N-q}>q$, where $N$ is the dimension of the domain $\Om\subset \R$. 
\item[-] We will use the somewhat standard notation for Stampacchia's truncation functions (see \cite{boc13}):
\[
T_k(s)=\max\{-k,\min\{s,k\}\}, \ \ \ G_k(s)=s-T_k(s), \quad \mbox{ for } k>0.
\]
\item[-] The letter $C$ will always denote a positive constant which may vary from line to line. 
\item[-] The Lebesgue measure of a set $A\subseteq \R^{N}$ is denoted by $|A|$.
\item[-] The symbol $\rightharpoonup$ denotes weak convergence. 
\item[-] The letter $\So$ denotes the best constant in Sobolev's inequality $\norma{u}{p^{*}}\le \So\norma{Du}{p}$, see \cite{talenti} for its value.
\end{itemize}
\section{The case $g(x)=0$ and $p\ge 2$.}\label{cs1}
In this section we analyze the Dirichlet problem:
\begin{equation} \label{gzero}
\begin{cases}
 -\lap u= \frac{\la}{|x|^{p}} |u|^{p-2}u + f \qquad & \mbox{in } \Omega,\\
u (x) = 0 & \mbox{on }  \partial \Omega,
\end{cases}
\end{equation}
when $p\ge 2$, $0<\la< \Ha^{p}$ and $f\in\L m$ for a suitable $m\ge1$.

We  begin by considering the ``truncated system'' system:
\begin{equation} \label{gzerot}
 -\Div(|D\un|^{p-2}D\un)  = \frac{\la}{|x|^{p}+\frac 1 n} |\un|^{p-2}\un + f_{n}
\end{equation}
where $f_{n}=T_{n}(f)$. The classical theory of Leray-Lions operators  guarantee the existence of a unique solution $\un\in\wpp\cap\L \infty$.
\subsection{High summability: $f\in\L{m}$ with $m\geq (p^{*})'$.}
\vspace{0.1in}
\begin{theorem}
Suppose $f\in\L m$ with $(p^{*})'\le m<\frac{N}{p}$ and :
\begin{equation}\label{lav}
0<\la < \frac{(m-1)N(N-mp)^{p-1}}{(p-1)^{p-1}m^{p}}
\end{equation}
Then the Dirichlet problem \eqref{gzero} has a solution $u\in\wpp\cap \L s$, where $s=\frac{mN}{N-pm}$.
\end{theorem}
\begin{proof}
 Take $\vf=\frac{1}{\g-p+1}|\un|^{\g-p}\un$ as a test function in \eqref{gzerot}, where $\g\ge p$ is a number to be chosen. If $\g=p$, we can easily see that $\norma{D\un}{p}\le C$, so $\un$ is bounded in $\wpp$ and hence $\un\wc u$ up to a subsequence. On the other hand, if $\g>p$ we have:
\[
 \iom |D \un|^{p}|\un|^{\g-p}\le\frac{\la}{\g-p+1}\iom \frac{|\un|^{\g}}{|x|^{p}}+\frac{1}{\g-p+1}\iom |f(x)||u_{n}|^{\g-p+1}.
\]
Rearranging using Hardy's inequality \eqref{hardy} with $v=|\un|^{\frac \g p}$
\[
 \frac {p^{p}} {\g^{p}}\iom |D |\un|^{\frac \g p} |^{p}\le \frac{\la}{\Ha^{p}(\g-p+1)}\iom  |D |\un|^{\frac \g p} |^{p}+\frac{1}{\g-p+1}\iom |f(x)||u_{n}|^{\g-p+1}.
\]
That is,
\[
 \left(\frac {p^{p}} {\g^{p}} -  \frac{\la}{\Ha^{p}(\g-p+1)} \right)\iom |D |\un|^{\frac \g p} |^{p}\le\frac{1}{\g-p+1}\norma{f}{m}\left(\iom |\un|^{(\g-p+1)m'} \right)^{\frac{1}{m'}}
\]
where we have used Holder's inequality in the last integral. 

Choosing $\g$ such that $\frac{\g p^{*}}{p}=(\g-p+1)m'=(p-1)s$, that is
\[
\g=\frac{(p-1)m(N-p)}{N-pm},
\]
we obtain:
\[
 \left(\frac {p^{p}} {\g^{p}} -  \frac{\la}{\Ha^{p}(\g-p+1)} \right)\left( \iom |\un|^{\frac{\g p^{*}}{p}} \right)^{\frac{p}{p^{*}}-\frac{1}{m'}=\frac{1}{s}} \le\frac{1}{\g-p+1}\norma{f}{m}
\]
In order for this to be meaningful, we must require $\g\ge p$ and $\frac {p^{p}} {\g^{p}} > \frac{\la}{\Ha^{p}(\g-p+1)}$, that is:
\[
(p^{*})'\le m<\frac{N}{p} \text{ and }0<\la < \frac{p^{p} \Ha^{p}(\g-p+1)} {\g^{p}}=\frac{(m-1)N(N-mp)^{p-1}}{(p-1)^{p-1}m^{p}}
\]
We conclude that
\end{proof}
\[
 \norma{\un}{s}\le C\norma{f}{m}
\]
for some $C>0$ that doesn't depend on $n$. Therefore, $\un\to u$ with $u\in\L{s}$. We can easily see that $u$ is a weak solution by passing the limit in \eqref{gzerot}.
\vspace{0.1in}
\begin{remark}
If $p=2$, we recover Theorem 2.1 in \cite{boc06}.
\end{remark}
\vspace{0.1in}
\begin{remark}
Notice the contrast between this case and the case $\lap u = f$ (treated for example in \cite{boc13}). In the latter, we only need $(p^{*})'\le m<\frac{N}{p}$, whereas in the former an additional restriction has to be made in order to have the same regularity.
\end{remark}
\vspace{0.1in}
\begin{remark}
The inequality \eqref{lav} is actually optimal as the example in \cite[Ex.~2.2]{boc06} shows. The author construct a radial example in the case $p=2$,$\la=\frac{(m-1)N(N-mp)^{p-1}}{(p-1)^{p-1}m^{p}}$ and shows that the solution $u\notin \L{s}$.
\end{remark}
\subsection{Low summability: $f\in\L{m}$ with $1<m< (p^{*})'$}
In this scenario, the right hand side in \eqref{gzero} is not in $\mathbf{W}^{-1,p'}(\Omega)$, hence we can't apply classical existence results despite the coercivity. 
\vspace{0.1in}
\begin{theorem}
Suppose $f\in\L m$ with $1<m< (p^{*})'$ and 
\begin{equation}\label{clam}
0<\la < \frac{(m-1)N(N-mp)^{p-1}}{(p-1)^{p-1}m^{p}}.
\end{equation}
Then the Dirichlet problem \eqref{gzero} has a distributional solution $u\in \mathbf{W}_{0}^{1,q}(\Om)$, where $q=\frac{pmN}{-Np+m(N(p-2)+p)}$.
\end{theorem}
\begin{proof}
 Take $\vf=[(t+|\un|)^{\g-p+1} - t^{\g-p+1}]\mathrm{sgn}(\un)$ as a test function in \eqref{gzerot}, where $\g\ge 1$ is a number to be chosen. We have:
\begin{equation}\label{sb}
 (\g-p+1) \iom |D \un|^{p}(t+|\un|)^{\g-p}\le\la \iom \frac{|\un|^{p-2}\un\vf}{|x|^{p}}+\iom |f(x)||(t+|\un|)^{\g-p+1} - t^{\g-p+1}|
\end{equation}
Rearranging this we obtain:
\[
\begin{split}
 \frac{(\g-p+1)p^{p}}{\g^{p}} \iom |D[(t+\un)^{\frac{\g}{p}}-t^{\frac{\g}{p}}]|^{p}&\le \la \iom \frac{|(t+\un)^{\frac{\g}{p}}-t^{\frac{\g}{p}}|^{p}}{|x|^{p}}+
\la\iom \frac{1}{|x|^{p}}\left( |\un|^{p-2}\un\vf -|(t+\un)^{\frac{\g}{p}}-t^{\frac{\g}{p}}|^{p}\right)\\&+\iom |f(x)||(t+|\un|)^{\g-p+1} - t^{\g-p+1}|\\
 \end{split}
\]
Using Hardy and Sobolev inequalities we get:
\[
\begin{split}
 \So^{p}\left(\frac{(\g-p+1)p^{p}}{\g^{p}} - \frac{\la}{\Ha^{p}}\right)\left(\iom |(t+\un)^{\frac{\g}{p}}|^{p^{*}}\right)^{\frac{p}{p^{*}}}&\le 
\la\iom \frac{1}{|x|^{p}}\left( |\un|^{p-2}\un\vf -|(t+\un)^{\frac{\g}{p}}-t^{\frac{\g}{p}}|^{p}\right)\\&+\iom |f(x)||(t+|\un|)^{\g-p+1} - t^{\g-p+1}|\\
 \end{split}
\]
If we fix $n$ and let $t\to0$ we obtain:
\[
\So^{p}\left(\frac{(\g-p+1)p^{p}}{\g^{p}} - \frac{\la}{\Ha^{p}}\right)\left(\iom |\un|^{\frac{\g p^{*}}{p}}\right)^{\frac{p}{p^{*}}}\le \norma{f}{m}\left(\iom |\un|^{(\g-p+1)m'} \right)^{\frac{1}{m'}}
\]
Now we choose $\g$ such that $\frac{\g p^{*}}{p}=(\g-p+1)m'$, that is we choose $$\g=(p-1)\frac{m(N-p)}{N-mp}.$$
We conclude using \eqref{clam} that
\[
\norma{\un}{s}\le C \norma{f}{m}
\]
where $s=\frac{mN}{N-pm}$.

Notice that by \eqref{sb} we also have for a fixed $t>0$:
\[
 \iom \frac{|D \un|^{p}}{(t+|\un|)^{p-\g}}\le C,
\]
hence, fix $\al>1$ we have:
\[
 \iom |D \un|^{\al} \le \iom \frac{|D \un|^{\al}}{(t+|\un|)^{\frac{(p-\g)\alpha}{p}}} (t+|\un|)^{\frac{(p-\g)\alpha}{p}},
\]
Using Holder Inequality with $\frac p \al$ and $\frac{p}{p-\al}$:
\[
 \iom |D \un|^{\al} \le C \left( \iom (t+|\un|)^{\frac{(p-\g)\alpha}{p-\al}}\right)^{\frac{p-\al}{p}},
\]
Let's choose $\al$ such that $\frac{(p-\g)\alpha}{p-\al}=s$ , that is ,
\[
\al=\frac{pmN}{-Np+m(N(p-2)+p)},
\]
which is positive and greater than 1, since we took $m<(p^{*})'$. 

We conclude that
\[
\norma{D\un}{\al}\le C,
\]
so up to a subsequence, $\un\wc u$ in $\mathbf{W}_{0}^{1,\al}(\Om)$.

We can easily see that this $u$ is a distributional solution, see for instance \cite{boc922}. (Notice that $\frac{|\un|^{p-1}\un}{|x|^{p}}\to \frac{|u|^{p-1}u}{|x|^{p}}$ in $\L1$).
\end{proof}
\vspace{0.1in}
\begin{remark}
If $p=2$, we recover Theorem 3.1 in \cite{boc06}.
\end{remark}
\section{The case $g(x)\ge 0$: Regularity gain}\label{cs3}
In this section we analyze the following problems:
\begin{equation} \label{gnot}
\begin{cases}
 -\lap u+g(x)u= \frac{\la}{|x|^{p}} |u|^{p-2}u + f \qquad & \mbox{in } \Omega,\\
u (x) = 0 & \mbox{on }  \partial \Omega,
\end{cases}
\end{equation}
when $0<\la< \Ha^{p}$ and $f,g \in\L 1$. 

Recall that the reason for choosing $0<\la< \Ha^{p}$ is to ensure the coercivity of the operator $T(u)=-\lap u-\frac{\la}{|x|^{p}} |u|^{p-2}u$.
In this case, the lower order term allow us to obtain not only a distributional solution but a weak solution in $\wpp$, where by a weak solution we mean $u\in\wpp$ such that $g(x)u\in\L1$ and for every $\vf\in\wpp\cap\L\infty$ we have:
\[
 \iom |Du|^{p-2}Du\cdot D\vf +\iom g(x)u\,\vf= \iom \frac{\la}{|x|^{p}} |u|^{p-2}u\,\vf + \iom f\,\vf
\]
\subsection{(possibly) unbounded solutions}
\vspace{0.1in}
\begin{theorem}\label{t4}
Suppose $p\ge 2$ and there is a $M>0$ such that
\begin{equation}
|f(x)|\le M g(x).
\end{equation}
Then the Dirichlet problem \eqref{gnot} has a weak solution $u\in\wpp$. Moreover, $u\in\L s$ for any $s\in[p^{*},s_{\la})$, where $s_{\la}$ is the unique solution to the following equation:
\[
\la=\frac{\Ha^{p}[p(\frac{s_{\la}}{p^{*}}-1)+1]}{(\frac{s_{\la}}{p^{*}})^{p}}.
\]
\end{theorem}
\begin{proof}
We consider the truncated equation
\begin{equation}\label{crop}
 -\lap \un+\frac{g(x)}{1+\frac{M}{n}g(x)}\un= \frac{\la}{|x|^{p}+\frac 1 n} \frac{|\un|^{p-2}\un}{1+\frac{|\un|^{p-1}}{n}} + \frac{f}{1+\frac{|f|}{n}}
\end{equation}
Classical Leray-Lions theory guarantees the existence of a weak solution $\un\in\wpp\cap\L\infty$, see for example Theorem 5.1 or Example 9.12 in \cite{boc13}.

Set $g_{n}=\frac{g(x)}{1+\frac{M}{n}g(x)}$ and $f_{n}=\frac{f}{1+\frac{|f|}{n}}$, by hypothesis we have:
\begin{equation}\label{fg}
|f_{n}(x)|\le M g_{n}(x).
\end{equation}
Now, take $\vf=\un$ as a test function in \eqref{crop}, we have:
\[
 \iom |Du|^{p}+\iom g_{n}\un^{2}\le \la\iom \frac{1}{|x|^{p}+\frac 1 n} \frac{|\un|^{p}}{1+\frac{|\un|^{p-1}}{n}} + \iom |f_{n}\un|
\]
Using \eqref{fg}, we have:
\[
 \iom |Du|^{p}+\iom g_{n}\un^{2}\le  \la\iom\frac{|\un|^{p}}{|x|^{p}} + M\iom g_{n}|\un|
\]
By Hardy's inequality:
\[
 \iom |Du|^{p}+\iom g_{n}\un^{2}\le  \frac{\la}{\Ha^{p}}\iom |D\un|^{p}+ M\iom g_{n}|\un|
\]
Rearranging, we obtain:
\[
\left(1-\frac{\la}{\Ha^{p}}\right) \iom |Du|^{p}\le  \iom g_{n}|\un|(M-\un).
\]
Using the trivial fact that the parabola $x(M-x)$ has a maximum value of $M^{2}$ if $0\le x\le M$, we conclude that:
\[
\left(1-\frac{\la}{\Ha^{p}}\right) \iom |Du|^{p}\le  M^{2}\iom g_{n}\le M^{2}\iom g.
\]
Therefore, $\un$ is bounded in $\wpp$ and hence $\un\wc u$ for some $u\in \wpp$. As before, we can easily pass the limit and conclude that $u$ is a solution. The term $g_{n}(x)\un$ is the tricky one, however Hardy's inequality and Vitali's convergence theorem guarantees the $\L1$ convergence.

We now prove the regularity of $u\in\wpp$. It's enough to prove that $\un$ is bounded in $\L s$. Take $\vf=|\un|^{p(\g-1)}\un$ as a test function in \eqref{crop}, for some $\g\ge 1$ to be chosen later. We have:
\[
 (p(\g-1)+1)\iom |Du|^{p}|\un|^{p(\g-1)}+\iom g_{n}|\un|^{p(\g-1)}\un^{2}\le \la\iom \frac{|\un|^{p\g}}{|x|^{p}} + \iom |f_{n}||\un|^{p(\g-1)+1}.
\]
Reasoning as before, we apply Hardy's inequality and use \eqref{fg} to obtain:
\[
 \left(\frac{p(\g-1)+1}{\g^{p}}-\frac{\la}{\Ha^{p}}\right)\iom |D|\un|^{\g}|^{p}+\iom g_{n}|\un|^{p(\g-1)+2}\le  M\iom g_{n}|\un|^{p(\g-1)+1}.
\]
Simplifying we get:
\[
 \left(\frac{p(\g-1)+1}{\g^{p}}-\frac{\la}{\Ha^{p}}\right)\iom |D|\un|^{\g}|^{p}\le \iom g_{n}|\un|^{p(\g-1)+1}(M-|\un|) .
\]
As before, we can clearly see that the function $x^{p(\g-1)+1}(M-x)$ has a maximum for $0\le x\le M$. We have:
\[
 \left(\frac{p(\g-1)+1}{\g^{p}}-\frac{\la}{\Ha^{p}}\right)\iom |D|\un|^{\g}|^{p}\le M^{p(\g-1)+2}\iom g .
\]
By Sobolev's inequality:
\[
 \left(\frac{p(\g-1)+1}{\g^{p}}-\frac{\la}{\Ha^{p}}\right)\frac{1}{\So}\iom |\un|^{\g p^{*}}\le M^{p(\g-1)+2}\iom g .
\]
In order for this to make sense, we must have:
\[
\la<\frac{\Ha^{p}[p(\g-1)+1]}{\g^{p}}
\]
Let $s=\g p^{*}$, and consider the function $h(s)=\frac{\Ha^{p}[p(\frac{s}{p^{*}}-1)+1]}{(\frac{s}{p^{*}})^{p}}$. Notice that if $s=p^{*}$ then $h(s)=\Ha^{p}$ and $\la<\Ha^{p}$ which is always true. Also, $h(s)$ is decreasing in $[p^{*},\infty)$, so eventually it will hit $\la$, at a point $s=s_{\la}$. In conclusion, the estimate is valid for $p^{*}\le s < s_{\la}$.
\end{proof}
\subsection{Bounded solutions: $1<p\le 2$}
\vspace{0.1in}
\begin{theorem}\label{ls}
Suppose there are $M>0, N>1$ such that
\begin{equation}\label{fgk}
|f(x)|\le M g(x) \text{ and } N\frac{\la}{|x|^{p}}\le g(x)
\end{equation}
If we assume $1<p\le 2$, then the Dirichlet problem \eqref{gnot} has a bounded weak solution $u\in\wpp\cap\L\infty$. 
\end{theorem}
\begin{proof}
We consider the truncated equation again (slightly modified):
\begin{equation}\label{crop2}
 -\lap \un+g_{n}(x)\un= \frac{\la}{|x|^{p}+\frac{MN\la}{n}} \frac{|\un|^{p-2}\un}{1+\frac{|\un|^{p-1}}{n}} + f_{n}
\end{equation}
Notice that 
\begin{equation} \label{relkl}
N\frac{\la}{|x|^{p}+\frac{MN\la}{n}} \le g_{n}(x).
\end{equation}
Now, fix $k>0$ and choose $\vf=G_{k}(\un)$ as a test function in \eqref{crop2}(See the Notation section for the definition of $G_{k}(s)$). We have:
\[
 \iom |DG_{k}(\un)|^{p}+\iom g_{n}(x)\un G_{k}(\un) \le \iom\frac{\la}{|x|^{p}+\frac{MN\la}{n}} |\un|^{p-1}|G_{k}(\un)| + \iom |f_{n}G_{k}(\un)|.
\]
Using \eqref{relkl},\eqref{fgk} and the fact that $sG_{k}(s)\ge 0$ we obtain:
\[
 \iom |DG_{k}(\un)|^{p}+\iom g_{n}(x)|\un G_{k}(\un)| \le \iom\frac{g_{n}(x)}{N} |\un|^{p-1}|G_{k}(\un)| + M\iom g_{n}(x)|G_{k}(\un)|.
\]
Simplifying:
\[
\iom g_{n}(x)|G_{k}(\un)|\left( |\un|(1-\frac{|\un|^{p-2}}{N})-M\right) \le 0
\]
Choose $k\gg0$ such that $ k(1-\frac{1}{k^{2-p}N})>M$, for that $k$ we must have
\[
|G_{k}(\un)|=0
\]
Which is to say that $\norma{\un}{\infty}\le k$. On the other hand, choosing $\vf=\un$ as a test function in \eqref{crop2}:
\[
 \iom |D\un|^{p}\le \iom\frac{\la}{|x|^{p}+\frac{MN\la}{n}} |\un|^{p}+ \iom |f_{n}\un|.
\]
Using Hardy's and Holder's inequalities:
\[
\left(1 - \frac{\la}{\Ha^{p}}\right) \iom |D\un|^{p}\le \norma{f}{1}\norma{\un}{\infty}.
\]
We conclude that  $\un$ is bounded in $\wpp\cap\L\infty$, hence up to a subsequence $\un\wc u$ in $\wpp$. Arguing as in the proof of theorem \ref{t4}, we can pass the limit in \eqref{crop2}, hence $u\in\wpp\cap\L\infty$ is a weak solution.
\end{proof}
\vspace{0.1in}
\begin{remark}
Notice that if $p>2$, the argument in the proof the theorem \ref{ls} fails and we believe it's possible that unbounded solutions may exist even with the strong restriction \eqref{fgk}.
\end{remark}
\vspace{0.1in}
\begin{remark}
We could increase the difficulty of the problem treated in this section if we had added a term of the form $g(x)u^{\theta-1}u$, the problem would still be solvable but considerably more difficult since the $\theta$ would interact with the $p$ in the estimates leading to existence estimates depending on different values of $\theta$ and $p$. We plan to address this question in future works.
\end{remark}
\vspace{0.1in}
\begin{remark}
The case $p=1$ remains challenging, that is, existence and regularity for the solution of the following problem:
\[
\Div(\frac{Du}{|Du|}) + g(x) u = \frac{\la}{|x|} \frac{u}{|u|} + f,
\]
where $f\in\L1$. We can't use Hardy's inequality, at least, not in its usual form. See \cite{chata24} for a related problem.
\end{remark}
\bibliography{sn-article}
\end{document}